\documentclass[12pt]{amsart}

\usepackage{pinlabel}
\usepackage[all]{xy}
\usepackage{amscd}
\usepackage{amssymb}
\usepackage{xcolor}
\usepackage{mathrsfs}

\newcommand{\bc}{}

\usepackage{url}
\usepackage{hyperref}

\title{Cappell-Shaneson knot pairs with the same Alexander polynomial}
\author[H.~Endo]{Hisaaki Endo}
\address{Department of Mathematics\\Institute of Science Tokyo\\
2-12-1 Oh-okayama\\\linebreak Meguro-ku\\Tokyo 152-8551\\Japan}
\email{endo@math.titech.ac.jp}
\author[K.~Iwaki]{Kazunori Iwaki}
\address{Department of Mathematics\\Institute of Science Tokyo\\
2-12-1 Oh-okayama\\\linebreak Meguro-ku\\Tokyo 152-8551\\Japan}
\email{iwaki.k.math@gmail.com}
\author[A.~Pajitnov]{Andrei Pajitnov}
\address{Laboratoire Math\'ematiques Jean Leray UMR 6629,
Universit\'e de Nantes,
Facult\'e des Sciences,
2, rue de la Houssini\`ere,
44072, Nantes, Cedex}  
\email{andrei.pajitnov@univ-nantes.fr}

\newcommand{\cs}{Cappell-Shaneson:1976-1}
\newcommand{\csm}{Cappell-Shaneson matrix}
\newcommand{\cskp}{Cappell-Shaneson knot pair}
\newcommand{\csss}{Cappell-Shaneson~ }

\theoremstyle{plain}
\newtheorem{theorem}{Theorem}[section]

\newtheorem{corollary}[theorem]{Corollary}

\theoremstyle{definition}
\newtheorem{definition}[theorem]{Definition}

\theoremstyle{remark}
\newtheorem{remark}[theorem]{Remark}

\newcommand{\ve}{\varepsilon}

\renewcommand{\t}{\theta}

\newcommand{\rr}{{\mathbb{R}}}

\newcommand{\ttt}{{\mathbb{T}}}

\newcommand{\zz}{{\mathbb{Z}}}

\newcommand{\PPPP}{{\mathscr{P}}}

\newcommand{\bere}{\begin{rema}}
\newcommand{\bede}{\begin{defi}}

\renewcommand{\beth}{\begin{theo}}
\newcommand{\bele}{\begin{lemm}}
\newcommand{\bepr}{\begin{prop}}
\newcommand{\beeq}{\begin{equation}}
\newcommand{\bega}{\begin{gather}}
\newcommand{\begaa}{\begin{gather*}}
\newcommand{\been}{\begin{enumerate}}

\newcommand{\bedee}{\begin{defii}}
\newcommand{\bethh}{\begin{theoo}}
\newcommand{\belee}{\begin{lemmm}}
\newcommand{\beprr}{\begin{propp}}

\newcommand{\beco}{\begin{coro}}

\newcommand{\beal}{\begin{aligned}}

\newcommand{\enre}{\end{rema}}

\newcommand{\enco}{\end{coro}}
\newcommand{\enpr}{\end{prop}}
\newcommand{\enth}{\end{theo}}
\newcommand{\enle}{\end{lemm}}
\newcommand{\enen}{\end{enumerate}}
\newcommand{\enga}{\end{gather}}
\newcommand{\engaa}{\end{gather*}}
\newcommand{\eneq}{\end{equation}}
\newcommand{\enal}{\end{aligned}}

\newcommand{\bq}{\begin{equation}}
\newcommand{\bqq}{\begin{equation*}}

\renewcommand{\geq}{\geqslant}

\newcommand{\sm}{\setminus}

\newcommand{\sbs}{\subset}

\newcommand{\wrt}{with respect to}

\newcommand{\nei}{neighbourhood}

\newcommand{\Prf}{{\it Proof.\quad}}

\newcommand{\chart}{\Phi_p:U_p\to B^n(0,r_p)}
\newcommand{\atlas}{\{\Phi_p:U_p\to B^n(0,r_p)\}_{p\in S(f)}}

\newcommand{\qs}{\hfill\square}

\theoremstyle{plain}

\newtheorem{theo}{Theorem}[section]
\newtheorem{lemm}[theo]{Lemma}
\newtheorem{prop}[theo]{Proposition}
\newtheorem{coro}[theo]{Corollary}

\theoremstyle{definition}

\newtheorem{defi}[theo]{Definition}
\newtheorem{rema}[theo]{Remark}

\newcommand{\arrh}[3]
{
\xymatrix{
{#1} \ar[r]^<<<<{#2}  &{#3}
}
}

\newcommand{\arrr}[1]
{\arrh {}{#1}{}}

\newcommand{\arr}
{\arrr {}}

\newcommand{\arrto}
{\xymatrix{{} \ar@{|-{>}}[r]  & {} } }

\newcommand{\arrinto}
{\xymatrix{{} \ar@{^{(}->}[r]  & {} } }

\begin{document}

\begin{abstract}
{\bc
It is well known that for $m\geq 2$
there are at most two non-equivalent $m$-knots with diffeomorphic
exterior. Such pair of knots will be called {\it non-reflexive knot pair}.
A classical problem in topology is to
determine all dimensions where such
knot pairs
exist. In 1976 Cappell and Shaneson gave a method
of constructing non-reflexive knot pairs. In the present paper we construct
an infinite family of new examples of \csss knot pairs, and give examples of
\csss knot pairs that have the same Alexander polynomial but are inequivalent.}

\end{abstract}
\maketitle

\newcommand{\stss}{\sim_*}



\section{Introduction}
In this paper, a $C^{\infty}$ embedding of $S^{m}$ into $S^{m+2}$ is called an $m$-knot.
Two $m$-knots are said to be equivalent if and only if there exists a self-diffeomorphism of $S^{m+2}$ that maps one knot onto the other. It is a classical problem in knot theory whether in given
dimension there are inequivalent knots with diffeomorphic
{\bc complements}.
For the case of classical knots ($m=1$), {\bc a theorem of } Gordon and Luecke \cite[Theorem 1]{Gordon-Luecke:1989} asserts that the knots are equivalent if the {\bc complements} are homeomorphic.
 Gluck \cite[Theorem 1.]{Gluck:1961}, Browder \cite[Corollary 3.]{Browder:1967}, Lashof and Shaneson \cite[Theorem 1.5.]{Lashof-Shaneson:1969} proved that for $m \ge 2$ there are at most $2$ inequivalent $m$-knots with diffeomorphic complements.

 A knot that is determined by its complement is called { \it reflexive}. In 1976, Cappell and Shaneson \cite{Cappell-Shaneson:1976-1} constructed the first examples of non-reflexive knots for $m=3,4$.
Soon after that, Gordon \cite{Gordon:1976} has constructed non-reflexive 2-knots.
In 1992, {\bc  Suciu } \cite{Suciu:1992} constructed
non-reflexive $m$-knots for every $m \equiv 3 {\rm \ or\ } 4  ({\rm mod} \ 8)$.
 In 1999, Gu and Jiang \cite{Gu-Jiang} constructed non-reflexive knots for $m=5,6$.

 {\bc
 The paper  \cite{Cappell-Shaneson:1976-1} by Cappell and Shaneson contains  a general method of constructing non-reflexive knots. Namely,
 to each $(n\times n)$-matrix  satisfying certain conditions
 (see Definition \ref{definition:Cappell-Shaneson matrix})
  they associated two inequivalent $(n-1)$-knots with homeomorphic complements (here $n\geq 4$). Such matrices will be called
  {\it \csss matrices}. These conditions turn out to be
 conditions on the characteristic polynomial of the matrix.
 A polynomial satisfying these conditions will be called
 {\it  Cappell-Shaneson  polynomial} and the pair of the two knots above will be called  {\it Cappell-Shaneson knot pair}. The present paper is the sequel
 of our work \cite{Endo-Iwaki-Pajitnov:2025} where we enumerated all
 \csss  polynomials  of degrees 4,5, and found new families of \csss polynomials
 of degrees 6 and 7.
}
However, it remains unknown whether Cappell-Shaneson polynomials exist for $n \ge 8$.

{\bc
In the present paper we prove that there are inequivalent \csss knot pairs with the same Alexander polynomial. (All previously known \csss knot pairs are determined by their Alexander polynomial.) We show in particular that there exist more than 10000 \csss knot pairs with Alexander polynomial equal to $x^5-99x^4+197x^3-197x^2+100x-1$. Furthermore we give new families of  \csss knot pairs for $n=4,5,6,7$. We describe  a complete algebraic invariant of Cappell-Shaneson knot pairs and provide a method to enumerate all knot pairs with given Alexander polynomial.
}

The present paper is organized as follows. In Section~\ref{section:geometry}, we first review Cappell and Shaneson's method to construct non-reflexive knots.
{\bc
Then we reduce the problem of classification of
\cskp s to a purely algebraic question concerning
classification of matrices.
 }

{\bc
This algebraic question is discussed in Section
\ref{section:classification_cs_matrix}.
To a given \csss matrix $A$ we associate
the ideal class monoid of the ring $\zz[\theta]$
where $\theta$ is a root of the characteristic polynomial of $A$.
We show that this invariant is a complete invariant of classes of
$*$-equivalence of matrices and therefore a complete invariant
of classes of equivalence of \csss knot pairs.

In Section~\ref{section:example} we compute ideal class monoids for several infinite families of
\csss polynomials in degrees 4,5,6,7. We give new examples of Cappell-Shaneson
knot pairs, and examples of inequivalent \csss knot pairs with same Alexander polynomial.
}

\subsection*{Acknowledgements}
The first author thanks the Nantes University, the DefiMaths program, and 
``Mission Invite'' for the support and warm hospitality. 
The first author is/was supported by JSPS KAKENHI Grant Numbers 
24K06707, 20K03578, 19H01788.
The third author was supported by JSPS International Fellowships for Research in Japan 
(Short-term), FY2022 Research Abroad and Invitation Program for International Collaboration, 
and the WRH program. 
This study was carried out using the TSUBAME4.0 supercomputer at Institute of Science Tokyo.

\section{Classification of Cappell-Shaneson knot pairs}\label{section:geometry}

In this section we prove a classification theorem for
 Cappell-Shaneson knot pairs.
It turns out  that the fundamental group of the complement is the complete invariant of such pairs, and these fundamental groups, in turn, are classified by equivalence classes of matrices
(with respect to a $*$-equivalence relation, see Definition
\ref{d:star-eq}).
We begin by a brief recollection of Cappell-Shaneson construction from \cite{\cs}.

\begin{definition}\label{definition:Cappell-Shaneson matrix}
{\bc A matrix }
$A \in \operatorname{SL}(n;\mathbb{Z})$ is called a \textit{Cappell-Shaneson matrix} if it satisfies the following condition $\mathrm{CS}_k$ for every integer $k$ in $\{ 1, \ldots, [n/2]\}$.

$\mathrm{CS}_k$: $\det(I- \bigwedge^k A)=\pm1$, where $I$ is the identity matrix and $\bigwedge^k A$ is the $k$-th exterior power of $A$.

Let $A$ be a Cappell-Shaneson matrix. We say $A$ is \textit{positive} if $(-1)^n \det(xI - A) > 0$ for every $x\in (-\infty,0)$.
\end{definition}

\noindent
{\bc
The condition $\mathrm{CS}_k$ is actually a condition on the characteristic polynomial
of the matrix (see \cite{Endo-Iwaki-Pajitnov:2025}, Corollary 2.8).
}

\begin{definition}
Let $A$ be a Cappell-Shaneson matrix. A monic {\bc polynomial}
$f(x)$ of degree $n$ is called a \textit{Cappell-Shaneson polynomial}
if it is the characteristic polynomial of a Cappell-Shaneson matrix of order $n$. A Cappell-Shaneson polynomial $f(x)$ is called \textit{positive} if $f(x)$ is the characteristic polynomial of a positive Cappell-Shaneson matrix.
\end{definition}
Let $A$ be a \csm~  in $\operatorname{SL}(n,\zz)$, and
$M_A$ be the mapping torus of the map $f_A:\ttt^{n}\to\ttt^{n}$ induced by $A$,
\begin{equation*}
	M_A=\ttt^n\times \mathbb{R}/(x,t)\sim(f_A(x),t-1)
	\end{equation*}

	\noindent
	We have  fibration of $M_A$ onto the circle $C$:
	\begin{equation}\label{f:fibr}
	 p_A:M_A\to \rr/\zz=C, \ \ p_A(x,t)=t
	 	\end{equation}
We denote by  $C_A$ (or just by $C$ if no confusion is possible)
the zero section of this fibration.
The normal bundle to $C$ in $M_A$ is trivial; choose a framing $\phi$ of $C$; we have then  a diffeomorphism of a tubular \nei~ $N_C$ of $C$ onto $D^n\times C$. The corresponding surgery replaces  $N_C$ by $S^{n-1}\times D^2$. It follows from the conditions $\rm{CS}_k$
that
the result of the surgery is a homotopy $(n+1)$-sphere $\Sigma^{n+1}$.
The image of $S^{n-1}\times\{0\}$ in $\Sigma^{n+1}$ is
 a $(n-1)$-knot, denote it
by $K(\phi)$. Choose a homeomorphism $F:\Sigma^{n+1}\to S^{n+1}$ which is smooth except maybe one point. Applying $F$ to the  knot $K(\phi)$ we obtain finally an $(n-1)$-knot in $S^{n+1}$. Its complement is diffeomorphic to $M_A\sm C$.

Since there are exactly two classes of equivalence of framings of $C$ in $M_A$ (corresponding to 2 elements of $\pi_1(SO(n))$)
the construction above leads to two $(n-1)$-knots which will
be denoted by $K_0(A)$ and $K_1(A)$. These knots have the following properties (see \cite{\cs}, Sections 3,4):

\been\item
The complement of $K_i(A)$ in $S^{n+1}$ is diffeomorhic
to $M_A\sm C$ (here $i = 0,1$).
\item
$K_0(A)$ is not equivalent to $K_1(A)$.
\item
 The  Alexander polynomial of $K_0(A)$ and $K_1(A)$ equals
the characteristic polynomial of  $A$.
\enen

\begin{definition}\label{d:pairs}
The unordered pair $\{K_0(A),K_1(A)\}$
will be  called a \textit{Cappell-Shaneson knot pair}
and denoted also by $\PPPP(A)$.
Two \cskp s
$\PPPP(A)$ and $\PPPP(B)$
are called {\it equivalent} if
$K_0(A)$ is equivalent to $ K_0(B)$ or to $ K_1(B)$
(then necessarily $K_1(A)$
is equivalent to $ K_1(B)$ or,  respectively, to $ K_0(B)$.
)
\end{definition}

Observe that for $i=0,1$  we have
\begin{equation}
 \label{f:semidir}\pi_1(S^{n+1}\setminus K_i(A))
\approx
\mathbb{Z}^{n}\rtimes_A \mathbb{Z},
\end{equation}

\noindent
Indeed,
the fundamental group of $S^{n+1}\setminus K_i(A)$
is isomorphic to that of $M_A\sm C$ and this space is the mapping torus of the map $f_A$ restricted to $\ttt^{n}\sm \{\omega\}$
(here $\omega$ is the image of the point $0\in\rr^n$ \wrt~
the projection $\rr^n\to\ttt^n$)
We have also
\begin{equation}
 \label{f:hh}
H_1(S^{n+1}\setminus K_i(A))
\approx
\zz.
\end{equation}

\noindent
Now we can proceed to classification of
\cskp s.
\begin{definition}\label{d:star-eq}
    Let $A,B$ in $\operatorname{GL}(n;\mathbb{Z})$.
    We say that $A$ is {\it  $*$-equivalent to $B$},
    and we write $A \stss B$,
    if $A$ is conjugate to $B$ or to $B^{-1}$ in
    $\operatorname{GL}(n;\mathbb{Z})$.
    \end{definition}

Denote $M_A\sm C$ by $Q_A$ for brevity.

\begin{theorem}\label{theorem:semidirect_product_classification}
    Let $A, B\in\operatorname{GL}(n;\mathbb{Z})$
    be \csss matrices.The following conditions are equivalent:
    \been
        \item $\pi_1(Q_A)\approx \pi_1(Q_B)$
        \item $A\stss B$.
    \enen
\end{theorem}
\Prf
Let $A\stss B$. If $A=DBD^{-1}$ with $D \in \operatorname{GL}(n;\mathbb{Z})$, then  the map
$$\mathbb{Z}^{n}\rtimes_A \mathbb{Z} \to  \mathbb{Z}^{n}\rtimes_B \mathbb{Z}; \ \  (x, y) \mapsto (D^{-1}x, y)$$
is an isomorphism.
If $A=DB^{-1}D^{-1}$ with $D \in \operatorname{GL}(n;\mathbb{Z})$, then  the map
$$\mathbb{Z}^{n}\rtimes_A \mathbb{Z} \to  \mathbb{Z}^{n}\rtimes_B \mathbb{Z}; \ \  (x, y) \mapsto (D^{-1}x, -y)$$
is an isomorphism.

Proceeding to the  implication $(2) \Rightarrow (1)$,
observe that the restriction of the fibration $p_A$ to
the subspace $Q_A$ is a fibration
$r_A:Q_A\to C$
 with fiber
$\ttt^{n}\sm \{*\}$. Consider the
 exact sequence of this fibration:
$$
0\arr
\zz^{n}
\arr
\pi_1(Q_A)
\arrr {(r_A)_*}
\zz
\arr
0.
$$
Let $\phi: \pi_1(Q_A) \to \pi_1(Q_B)$ be an isomorphism.
Since $H_1(Q_A)\approx H_1(Q_B)\approx \zz$,
there are only two epimorphisms $\pi_1(Q_B)\to\zz$ and
 only two epimorphisms $\pi_1(Q_A)\to\zz$.
Therefore we have a commutative diagram

$$
\xymatrix{0 \ar[r] & \zz^{n}\ar[d]^{{\phi}_0}\ar[r] &
\pi_1(Q_A) \ar[d]^{\phi}  \ar[r]^{~  (r_A)_*} & \zz
\ar[d]^{\varepsilon}  \ar[r]  & 0\\
0 \ar[r] \ar[r] & \zz^{n}\ar[r] & \pi_1(Q_B)  \ar[r]^{(r_B)_*} & \zz \ar[r] & 0
}
$$
where $\ve$
equals $1$ or $-1$ and $\phi_0=\phi~|~\zz^{n}$.
Assume that $\ve=1$.
Choose an element
$\t_A
\in
\pi_1(Q_A)$
such that
$(r_A)_*(\t_A)=1$.Then $(r_B)_*(\phi(\t_A))=1$
and the action of $\phi(\t_A)$ on $\zz^{n+1}$
equals the action of the matrix $B$.
We have therefore
$$
\phi_0(\t_A h \t_A^{-1} ) =
\phi(\t_A)\phi_0(h)\phi(\t_A)^{-1}$$
for any $h\in\zz^{n}$
so that $\phi_0(Ah)= B\phi_0(h)$ and $A$ is conjugate to $B$ via $\phi_0$.

In the case $\ve=-1$,
the argument  similar to the above shows that
the matrix $A$ is conjugate to $B^{-1}$. $\qs$

\begin{theorem}
    Cappell-Shaneson knot pairs $\PPPP(A)$
    and $\PPPP(B)$
    are equivalent if and only if the matrices $A$ and $B$ are $*$-equivalent.
\end{theorem}

\Prf
If $\{K_0(A), K_1(A)\} $ is equivalent to $\{K_0(B), K_1(B)\} $,
then  all these four knots have diffeomorphic complements,
therefore the fundamental groups of the complements are isomorphic,
and Theorem~\ref{theorem:semidirect_product_classification} implies that $A$ is $*$-equivalent to $B$.

To prove the inverse implication, observe that if
$A=DBD^{-1}$ with $D\in \operatorname{GL}(n;\mathbb{Z})$
 then a map
    $h: M_A \to M_B$ defined by the  formula
    $$(x,t)\mapsto (D^{-1}x,t)$$
    is a well-defined diffeomorphism sending the circle
    $C_A\sbs M_A$
    to the circle $C_B\sbs M_B$
    Therefore the exteriors of the  four knots $K_0(A), K_1(A), K_0(B), K_1(B)$ are diffeomorphic. The knots
    $K_0(A)$ and $ K_1(A)$ are not equivalent. For $n > 2$ there are at most $2$ inequivalent $(n-1)$-knots with diffeomorphic complements (see \cite{Gluck:1961}, \cite{Browder:1967} and \cite{Lashof-Shaneson:1969}),
    so $K_0(A)$ is equivalent to $K_0(B)$ or to $ K_1(B)$,
    and $\PPPP(A)$ is equivalent to $\PPPP(B)$.

    Furthermore, observe that $\PPPP(A)$ is equivalent to
    $\PPPP(A^{-1})$. Indeed, the formula
    $k(x,t)=(x,1-t)$
    determines a diffeomorphism from $M_A$ to $M_{A^{-1}}$
    sending the circle $C_A$ to the circle $C_{A^{-1}}$,
    and applying the same argument as above we deduce that
    $\PPPP(A)\sim\PPPP(A^{-1})$. Therefore if $A$ is
    conjugate to $B^{-1}$, we have
    $$\PPPP(A)\sim\PPPP(B^{-1})\sim \PPPP(B),$$
    and the proof of  Theorem is complete. $\qs$

\section{Classification of Cappell-Shaneson matrices}\label{section:classification_cs_matrix}
    In this section, we introduce a method to obtain $*$-equivalence classes of Cappell-Shaneson matrices when a Cappell-Shaneson polynomial is given.

    In the beginning, we collect some results for Cappell-Shaneson polynomials.
    
    \begin{theorem}[{\cite[Theorem 3.2.]{Gu-Jiang}}]\label{theorem:irr}
     Cappell-Shaneson polynomials are irreducible over $\mathbb{Z}$.
    \end{theorem}
    
    \begin{theorem}\label{theorem:irr2}
    Let $f(x) \in \mathbb{Z}[x]$ be an irreducible monic polynomial with $deg(f) =n$. For $A \in \operatorname{M}(n;\mathbb{Z})$, $f$ is the characteristic polynomial of $A$ if and only if $f(A)=O$.
    \end{theorem}
    
    \begin{proof}
        If $f$ is the characteristic polynomial of $A$, $f(A)=O$ is true by the Cayley-Hamilton theorem. We assume $f(A)=O$. Then, the minimal polynomial of $A$ divides $f$. The minimal polynomial of $A$ is equal to $f$ since $f$ is irreducible. The minimal polynomial of $A$ is equal to the characteristic polynomial of $A$ since both polynomials are monic polynomials with the same degree. Therefore, $f$ is the characteristic polynomial of $A$.
    \end{proof}

    By Theorem~\ref{theorem:irr2}, if a polynomial $f$ is a Cappell-Shaneson polynomial and $f(A)=O$ for $A \in \operatorname{M}(n;\mathbb{Z})$, then $A$ is a Cappell-Shaneson matrix. Therefore, $\{A \in \operatorname{M}(n;\mathbb{Z})\mid f(A)=O\}$ is the set of Cappell-Shaneson matrices whose characteristic polynomial is $f$.

    Next, we introduce an algebraic object called an ideal class monoid to classify Cappell-Shaneson matrices.

    \begin{definition}
    Let $R$ be an integral domain and $\mathcal{I}(R)$ the set of non-zero ideals. We define an equivalence relation $\approx$ on $\mathcal{I}(R)$ such that $I \approx J$ if and only if there exist non-zero $\alpha,\beta \in R$ such that $\alpha I = \beta J$. The general multiplication of ideals induces an operation in $\mathcal{I}(R)/\approx$, and this set becomes a commutative monoid. We call it \textit{ideal class monoid} and write $C(R)$. We say that an ideal $I$ is \textit{invertible} if $[I]$ is an invertible element in $C(R)$.
    \end{definition}

    \begin{theorem}[Latimer-MacDuffee \cite{Latimer-MacDuffee}, Taussky \cite{Taussky}]\label{theorem:LMT}
     Let $f(x) \in \mathbb{Z}[x]$ be an irreducible monic polynomial. Let $\theta$ be a root of $f(x)$. The following two maps are bijective and $\phi_f^{-1} = \psi_f$.
     \begin{itemize}
         \item Let $v=[v_1, \ldots, v_n]^\mathrm{T}\in\mathbb{Z}[\theta]^n$ be an eigenvector of $A$ and $S$ a $\mathbb{Z}$-submodule of $\mathbb{Z}[\theta]$ generated by $v_1,\ldots,v_n$. Define $\phi_f : \{A \in \operatorname{M}(n;\mathbb{Z})\mid f(A)=O\}/\sim_C \rightarrow C(\mathbb{Z}[\theta])$ as $\phi_f([A])=[S]$.
         \item Let $\{v_1, \ldots, v_n\}$ be a basis of an ideal $S$ and $v=[v_1,\ldots,v_n]^\mathrm{T}$. There is $A\in \operatorname{M}(n;\mathbb{Z})$ such that $\theta v=A v$. Define $\psi_f : C(\mathbb{Z}[\theta]) \rightarrow \{A \in \operatorname{M}(n;\mathbb{Z})\mid f(A)=O\}/\sim_C$ as $\psi_f([S])=[A]$.
     \end{itemize}
    \end{theorem}

    \begin{corollary}
        Let $f(x) \in \mathbb{Z}[x]$ be a Cappell-Shaneson polynomial. Let $\theta$ be a root of $f(x)$. Then, there is a one-to-one correspondence between $C(\mathbb{Z}[\theta])$ and $\{A\mid f(A)=O\}/\sim_C$.
    \end{corollary}

    The above discussion gives us how to classify Cappell-Shaneson matrices based on conjugacy. Furthermore, we use this formulation to classify $*$-equivalence classes.

\begin{theorem}\label{theorem:csmulinverse}
Let $A \in \operatorname{SL}(n;\mathbb{Z})$ be a Cappell-Shaneson matrix whose characteristic polynomial is $f$. Then $A^{-1}$ is a Cappell-Shaneson matrix whose characteristic polynomial is $(-x)^nf(x^{-1})$. We write this polynomial as $f^*$ and call it the \textit{signed reciprocal polynomial} of $f$.
\end{theorem}
\begin{proof}
Let $A \in \operatorname{SL}(n;\mathbb{Z})$ be a Cappell-Shaneson matrix. There exists $A^{-1} \in \operatorname{SL}(n;\mathbb{Z})$. 
\begin{align*}
    \det(I-\bigwedge^k A^{-1})&=\det(I-(\bigwedge^k A)^{-1}) \\
                            &= \det(\bigwedge^k A)\det(I-(\bigwedge^k A)^{-1})\\
                            &=(-1)^{{}_{n} C_k}\det(I-\bigwedge^k A)=\pm 1
\end{align*}
since $\det(\bigwedge^k A)=1$ and $\det(I-\bigwedge^k A) = \pm 1$.
So, $A^{-1}$ is a Cappell-Shaneson matrix. In addition, $A^{-1}$ satisfies positivity since 
\begin{align*}
    (-1)^n\det(xI-A^{-1})&= \det(A)\det(xI-A^{-1})\\
                   &=(-x)^{n}(-1)^n\det(x^{-1}I-A).
\end{align*}
\end{proof}

\begin{remark}
    Let $A \in \operatorname{SL}(n;\mathbb{Z})$ be a Cappell-Shaneson matrix whose characteristic polynomial is $f$. When $n \ge 4$, $f \ne f^*$.
\end{remark}
\begin{proof}
    If $f = f^*$, the characteristic polynomial of $A$ is equal to that of $A^{-1}$. For an eigenvalue $\lambda$ of $A$, $\lambda^{-1}$ is also an eigenvalue of $A$. If $\lambda^{-1}=\lambda$, $\lambda=\pm 1$. This contradicts irreducibility of the characteristic polynomial. Therefore, $\lambda^{-1} \ne \lambda$ and $1$ is an eigenvalue of $\bigwedge^2 A$. This contradicts the condition $\mathrm{CS}_2$ of the definition of the Cappell-Shaneson matrices for $n \ge 4$.
\end{proof}

For two Cappell-Shaneson matrices $A, B$, let $f_A,f_B$ be the characteristic polynomial of $A,B$, respectively. If $A \sim_* B$, then $f_A = f_B$ or $f_A=f_B^*$. 

\begin{theorem}    
    Let $f(x) \in \mathbb{Z}[x]$ be a Cappell-Shaneson polynomial. Let $\theta$ be a root of $f(x)$. Then, there is a one-to-one correspondence between $C(\mathbb{Z}[\theta])$ and $\{A\mid f(A)=O\}\cup\{A\mid f^*(A)=O\}/\sim_*$. The following two maps are bijective and $\overline{\phi_f}^{-1} = \overline{\psi_f}$.
     \begin{itemize}
         \item Let $v=[v_1, \ldots, v_n]^\mathrm{T}\in\mathbb{Z}[\theta]^n$ be an eigenvector of $A$ corresponding to $\theta$ and $S$ a $\mathbb{Z}$-submodule of $\mathbb{Z}[\theta]$ generated by $v_1,\ldots,v_n$. Define $\overline{\phi_f} : \{A\mid f(A)=O\}\cup\{A\mid f^*(A)=O\}/\sim_* \rightarrow C(\mathbb{Z}[\theta])$ as $\overline{\phi_f}([A])=[S]$. 
         \item Let $\{v_1, \ldots, v_n\}$ be a basis of an ideal $S$ and $v=[v_1,\ldots,v_n]^\mathrm{T}$. There is $A\in \operatorname{M}(n;\mathbb{Z})$ such that $\theta v=A v$. Define $\overline{\psi_f} : C(\mathbb{Z}[\theta]) \rightarrow \{A\mid f(A)=O\}\cup\{A\mid f^*(A)=O\}/\sim_*$ as $\overline{\psi_f}([S])=[A]$.
     \end{itemize}
\end{theorem}
\begin{proof}
    $\phi_{f^*}$ is a bijection between $\{A\mid f^*(A)=O\}/\sim_C$ and $C(\mathbb{Z}[\theta^{-1}])$ since $\theta^{-1}$ is a root of $f^*$. When $f(x)=x^n+c_{n-1}x^{n-1}+\cdots+c_1x+(-1)^n$, $\theta(c_{n-1}\theta^{n-1}+\cdots+c_1\theta)=(-1)^{n+1}$ holds since $f(\theta)=0$. So, $\theta^{-1}=(-1)^{n+1}(c_{n-1}\theta^{n-1}+\cdots+c_1\theta)\in \mathbb{Z}[\theta]$. $\theta \in \mathbb{Z}[\theta^{-1}]$ also holds by the same argument. Therefore, $\mathbb{Z}[\theta]=\mathbb{Z}[\theta^{-1}]$ and $\phi_{f^*}$ is a bijection between $\{A\mid f^*(A)=O\}/\sim_C$ and $C(\mathbb{Z}[\theta])$. 

    For $[B] \in \{A\mid f^*(A)=O\}/\sim_C$, let $v=[ v_1, \ldots, v_n ]^\mathrm{T}$ be an eigenvector of $B$ corresponding to $\theta$ such that all elements are included in $\mathbb{Z}[\theta]$. $\phi_f$ maps $[B]$ to an ideal class $[\langle v_1, \ldots, v_{n+1} \rangle_\mathbb{Z}]$. On the other hand, $v$ is an eigenvector of $B^{-1}$ corresponding to $\theta^{-1}$.  $\phi_{f^*}$ maps $[B^{-1}]$ to an ideal class $[\langle v_1, \ldots, v_{n+1} \rangle_\mathbb{Z}]$. So, $\overline{\phi_f}$ is well-defined.

    See Theorem~\ref{theorem:LMT} for bijectivity.
\end{proof}

    Therefore, the classification of Cappell-Shaneson matrices is equivalent to that of ideal classes. If we can compute the corresponding ideal class monoid explicitly, we can classify Cappell-Shaneson matrices for a given Cappell-Shaneson polynomial. We can sometimes compute it if $\mathbb{Z}[\theta]$ is integrally closed. In that case, the ideal class monoid is an ideal class group. We can use many methods from the number theory. 
    \begin{definition}
    Let $R$ be an integral domain and $K$ the fraction field of $R$. We say that $R$ is \textit{integrally closed} if and only if the following condition holds: if $\alpha \in K$ is a root of a monic polynomial in $R[x]$, $\alpha \in R$.
    \end{definition}
    
    \begin{theorem}[{\cite[Theorem 4.4]{Kim-Yamada:2017-1}}]
    If $\mathbb{Z}[\theta]$ is integrally closed, $C(\mathbb{Z}[\theta])$ is a ideal class group. 
    \end{theorem}

    Besides, we introduce Kummer-Dedekind criterion to consider Cappell-Shaneson matrices that are not $*$-equivalent to companion matrices.
    \begin{theorem}[{\cite[Theorem 8.2]{Stevenhagen:2008-1}}]\label{theorem:kummerdedekind}
        Let $p$ be a prime number and $f$ an irreducible monic polynomial. We denote by $\theta$ a root of $f$. Let us suppose $\bar{f} \in \mathbb{F}_p[x]$ can be factorized into $\Pi_i \bar{g_i}^{e_i}$. Pick $g_k$ which is a monic polynomial. Let $r_k \in \mathbb{Z}[x]$ be a remainder when $f$ is divided by $g_k$ in $\mathbb{Z}[x]$. Then, an ideal $(p, g_k(\theta))$ is invertible if and only if $e_k=1$ or $p^2 \nmid r_k$ in $\mathbb{Z}[x]$.
    \end{theorem}

    \begin{corollary}\label{corollary:sourceofinffamily}
    Let $p$ be a prime number and $f$ an irreducible monic polynomial with $deg(f) = n$. We denote by $\theta$ a root of $f$. Let us suppose $\bar{f} \in \mathbb{F}_p[x]$ can be factorized into $\Pi_i \bar{g_i}^{e_i}$. Let us suppose $g_k(x)=x-b$. Let $r_k \in \mathbb{Z}[x]$ be a remainder when $f$ is divided by $x - b$ in $\mathbb{Z}[x]$. If $e_k \neq 1$ and $p^2 \nmid r_k$ in $\mathbb{Z}[x]$, then $(p, \theta - b)$ is not invertible and $\{p, \theta - b, \theta ( \theta - b), \ldots , \theta^{n-2} (\theta - b)\}$ is a $\mathbb{Z}$-basis of $(p, \theta -b)$.
    \end{corollary}
    \begin{proof}
        $(p, \theta - b)$ is not invertible by Theorem~\ref{theorem:kummerdedekind}. Let $\langle p, \theta - b, \theta( \theta - b)  , \ldots , \theta^{n-2} (\theta - b) \rangle$ be a set of all linear combinations, which is a subset of $(p, \theta - b)$. 

        We show $p \theta^{j}, (\theta-b) \theta^{j} \in \langle p, \theta - b, \theta ( \theta - b), \ldots , \theta^{n-2} (\theta - b) \rangle$ for $j \in \{1, \ldots, n-1\}$ to show $(p, \theta - b)$ is a subset of $\langle p, \theta - b, \theta ( \theta - b), \ldots , \theta^{n-2} (\theta - b) \rangle$. If $j=1, \ldots , n-2$, $(\theta-b) \theta^{j} \in \langle p, \theta - b, \theta ( \theta - b), \ldots , \theta^{n-2} (\theta - b) \rangle$ for $j \in \{1, \ldots, n-1\}$ clearly. Let us suppose $f(x) = q(x)(x-b)+r_k$ where $q(x) \in \mathbb{Z}[x]$.
        \[
            (\theta -b) \theta^{n-1} = -(\theta - b)(q(x)-\theta^{n-1}) + r_k
        \] since $f(\theta)=0$. $(\theta - b) \theta^{n-1}$ is an element of $\langle p, \theta - b, \theta ( \theta - b), \ldots , \theta^{n-2} (\theta - b) \rangle$ since $deg(q(x)-\theta^{n-1}) \le n-2$ and $p\mid r_k$. We realize that $p \theta^{j} \in \langle p, \theta - b, \theta ( \theta - b), \ldots , \theta^{n-2} (\theta - b) \rangle$ for $j \in \{1, \ldots, n-1\}$ through the Taylor expansion of $\theta^j$ at $x=b$.
    \end{proof}

    This corollary shows how to find non-invertible ideal classes but can be used to find Cappell-Shaneson matrices that are not $*$-equivalent to companion matrices. That is because matrices corresponding to non-invertible ideal classes are not $*$-equivalent to companion matrices. Companion matrices correspond to the identity element by Theorem~\ref{theorem:LMT} and it is trivially invertible.
    
\section{Examples of new Cappell-Shaneson knot pairs}\label{section:example}
    In this section, we give new examples of non-reflexive $(n-1)$-knots using Cappell and Shaneson's method for $n=4,5,6,7$. First, we choose a Cappell-Shaneson polynomial as an Alexander polynomial of non-reflexive knots. Next, we classify ideal classes corresponding to it. Lastly, we obtain non-reflexive knots from Cappell-Shaneson matrices corresponding to the ideal classes.
\subsection{Cappell-Shaneson knot pairs for \texorpdfstring{$n=4$}{n=4}}
    We give new Cappell-Shaneson knot pairs for $n=4$. We start from a positive Cappell-Shaneson polynomial $x^4+ax^3+(-2a-2)x^2+(a-1)x+1$ with $a \le 0$, which is newly found in our paper \cite{Endo-Iwaki-Pajitnov:2025}. We want to compute the ideal class monoid $C(\mathbb{Z}[\theta_a])$, where $\theta_a$ is a root of $x^4+ax^3+(-2a-2)x^2+(a-1)x+1$. We can know when $C(\mathbb{Z}[\theta_a])$ is an ideal class group and when it is not trivial by the following code for SageMath\cite{Sage:2019-1}. See Figure~\ref{figure:cs4} for the result of this snippet.
    \begin{verbatim}
    for a in range(-100, 1):
        f = x^4 + a*x^3+(-2*a-2)*x^2+(a-1)*x+1
        K.<theta> = NumberField(f)
        O = K.order(theta)
        if O.is_integrally_closed():
            G = K.class_group()
            print(a, G.order())
    \end{verbatim}

    \begin{figure}[htbp]
    \centering
    \includegraphics[scale=0.6]{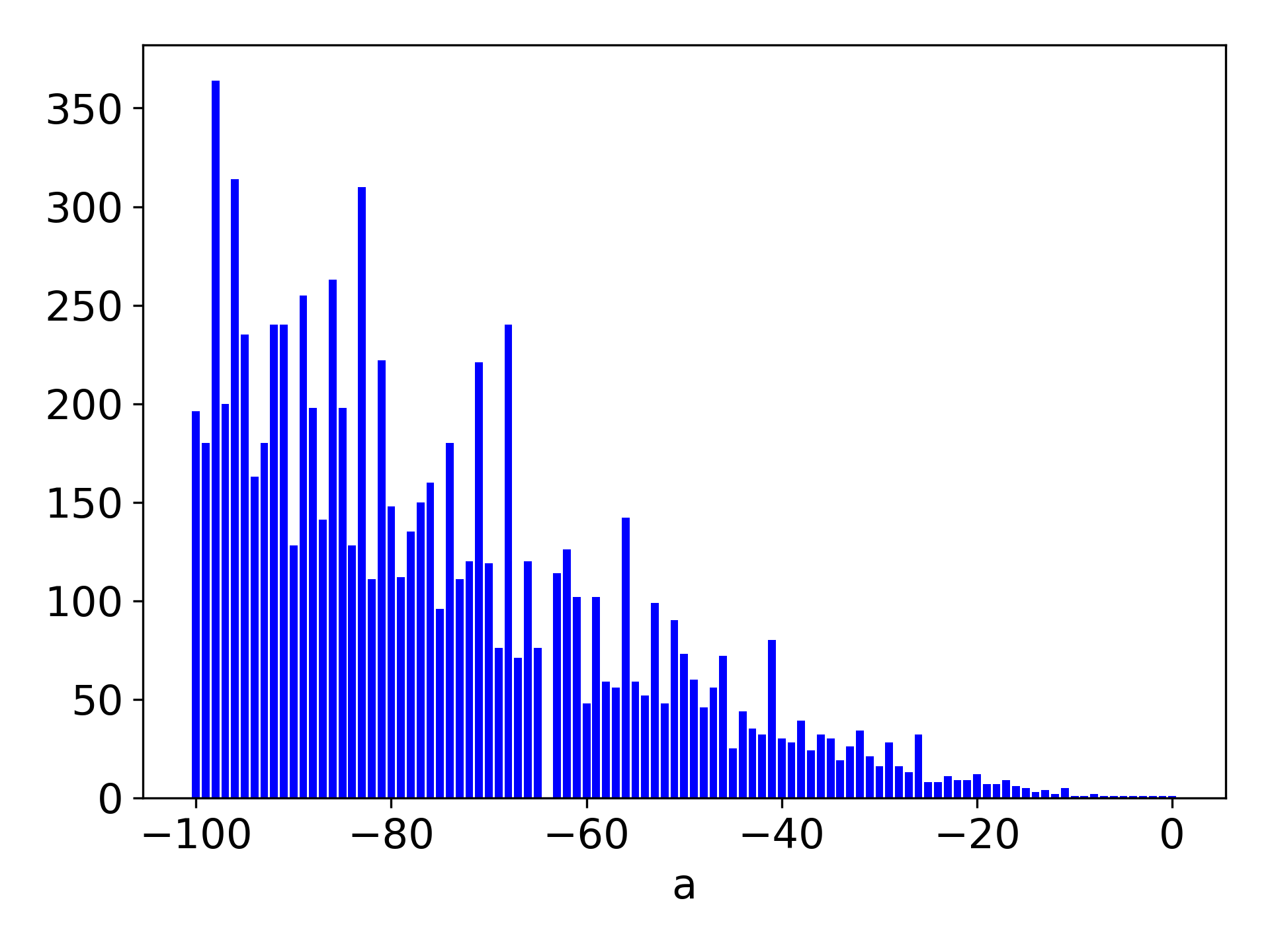}
    \caption{Orders of ideal class groups corresponding to $x^4+ax^3+(-2a-2)x^2+(a-1)x+1$.}
    \label{figure:cs4}
    \end{figure}

    The bar is missing if $C(\mathbb{Z}[\theta_a])$ is not a group for the value. $a=-64$ is the only value in this figure for which $C(\mathbb{Z}[\theta_a])$ is not a group. The y-axis represents the order of $C(\mathbb{Z}[\theta_a])$, which is the number of Cappell-Shaneson knot pairs. Those knot pairs are not equivalent to known Cappell-Shaneson knot pairs since the first Alexander polyomials are different. In addition, we realize that there are many distinct Cappell-Shaneson knot pairs which are not distinguishable by the first Alexander polynomial. $a=-8$ is the largest value that $C(\mathbb{Z}[\theta_a])$ is not trivial. The order of $C(\mathbb{Z}[\theta_{-8}])$ is equal to $2$. This means that there exist two inequivalent Cappell-Shaneson knot pairs whose Alexander polynomials are equal to $x^4-8x^3+14x^2-9x+1$. One pair is obtained by the companion matrix 

    \[
        \begin{bmatrix}
        0 & 1 & 0 & 0\\
        0 & 0 & 1 & 0\\
        0 & 0 & 0 & 1\\
        -1 & 9 & -14 & 8\\    
        \end{bmatrix}.
    \]

    Another pair is obtained from 

    \[
        \begin{bmatrix}
        2 & 3 & 0 & 0\\
        2 & 4 & 1 & 0\\
        0 & 1 & 1 & 1\\
        1 & 2 & 0 & 1\\    
        \end{bmatrix}.
    \]
    This matrix corresponds to $[(3, \theta_{-8}^3 -7 \theta_{-8}^2 +7 \theta_{-8})]$, which is the generator of $C(\mathbb{Z}[\theta_{-8}])$. 
    
    Note that $a=-25$ is the largest value for which the ideal class group $C(\mathbb{Z}[\theta_a])$ has a structure other than the cyclic group, $C(\mathbb{Z}[\theta_{-25}])=C_4\times C_2$. In this case, there are $8$ inequivalent knot pairs that cannot be distinguished by the Alexander polynomial.

    In addition, we give infinitely many examples of Cappell-Shaneson knot pairs that are not distinguishable by the first Alexander polynomial.

\begin{theorem}
    The following two matrices are positive Cappell-Shaneson matrices for any $l \le 0$. These matrices are not $*$-equivalent and give us inequivalent Cappell-Shaneson knot pairs with the same Alexander polynomial:
    \[
        \begin{bmatrix}
    		0 & 1 & 0 & 0 \\
    		0 & 0 & 1 & 0 \\
    		0 & 0 & 0 & 1\\
            -1 & -121l+65 & 242l-126 & -121l+64\\
    		\end{bmatrix},
    \]
      \[
        \begin{bmatrix}
    		2 & 11 & 0 & 0 \\
    		0 & 0 & 1 & 0 \\
    		0 & 0 & 0 & 1\\
            -22l+11 & -121l+61 & -2 & -121l+62\\
    		\end{bmatrix}.
    \]

\end{theorem}
\begin{proof}
    Define $f_a(x)$ as $x^4+ax^3+(-2a-2)x^2+(a-1)x+1$, which is a positive Cappell-Shaneson polynomial when $a \le 0$. Let $\theta_a$ be a root of $f_a(x)$. Then, $(11, \theta_{11^2l-64} - 2)$ is not invertible and 
    \[
    \begin{aligned}
    \{\, 11, \ \theta_{11^2l-64} - 2, \ \theta_{11^2l-64} ( \theta_{11^2l-64} - 2), \ \theta_{11^2l-64}^2 (\theta_{11^2l-64} - 2)\, \}
    \end{aligned}
    \]
    is a $\mathbb{Z}$-basis of $(11, \theta_{11^2l-64} - 2 )$ for any $l \le 0$ since
    \begin{align*}
        f_{11^2l-64}(x)&\equiv (x-2)^2(x^2+6x+3) \pmod {11} ,\\
        f_{11^2l-64}(x) &= (x-2)(x^3+(121l-62)x^2+2x+(121l-61))+242l-121
    \end{align*}
    and Corollary~\ref{corollary:sourceofinffamily} can be used. The following matrix corresponds to the ideal class of $(11, \theta_{11^2l-64} - 2)$,
      \[
        \begin{bmatrix}
    		2 & 11 & 0 & 0 \\
    		0 & 0 & 1 & 0 \\
    		0 & 0 & 0 & 1\\
            -22l+11 & -121l+61 & -2 & -121l+62\\
    		\end{bmatrix}.
    \]
    This matrix is not $*$-equivalent to the companion matrix of $f_{11^2l-64}(x)$ since the ideal class corresponding to the companion matrix is invertible. Therefore, the above matrix and the companion matrix give us inequivalent knot pairs that cannot be distinguished by the Alexander polynomial.
\end{proof}
\subsection{Cappell-Shaneson knot pairs for \texorpdfstring{$n=5$}{n=5}}
    We give new Cappell-Shaneson knot pairs for $n=5$. We use a positive Cappell-Shaneson polynomial $x^5+ax^4+(-2a-1)x^3+(2a+1)x^2+(-a+1)x-1$ with $a \le 0$, which is newly found in our paper \cite{Endo-Iwaki-Pajitnov:2025}. We can compute the orders of ideal class groups corresponding to $x^5+ax^4+(-2a-1)x^3+(2a+1)x^2+(-a+1)x-1$ in the same way as for $n=4$. The order is the number of Cappell-Shaneson knot pairs with $x^5+ax^4+(-2a-1)x^3+(2a+1)x^2+(-a+1)x-1$ as the first Alexander polynomial.

    \begin{figure}[htbp]
    \centering
    \includegraphics[scale=0.6]{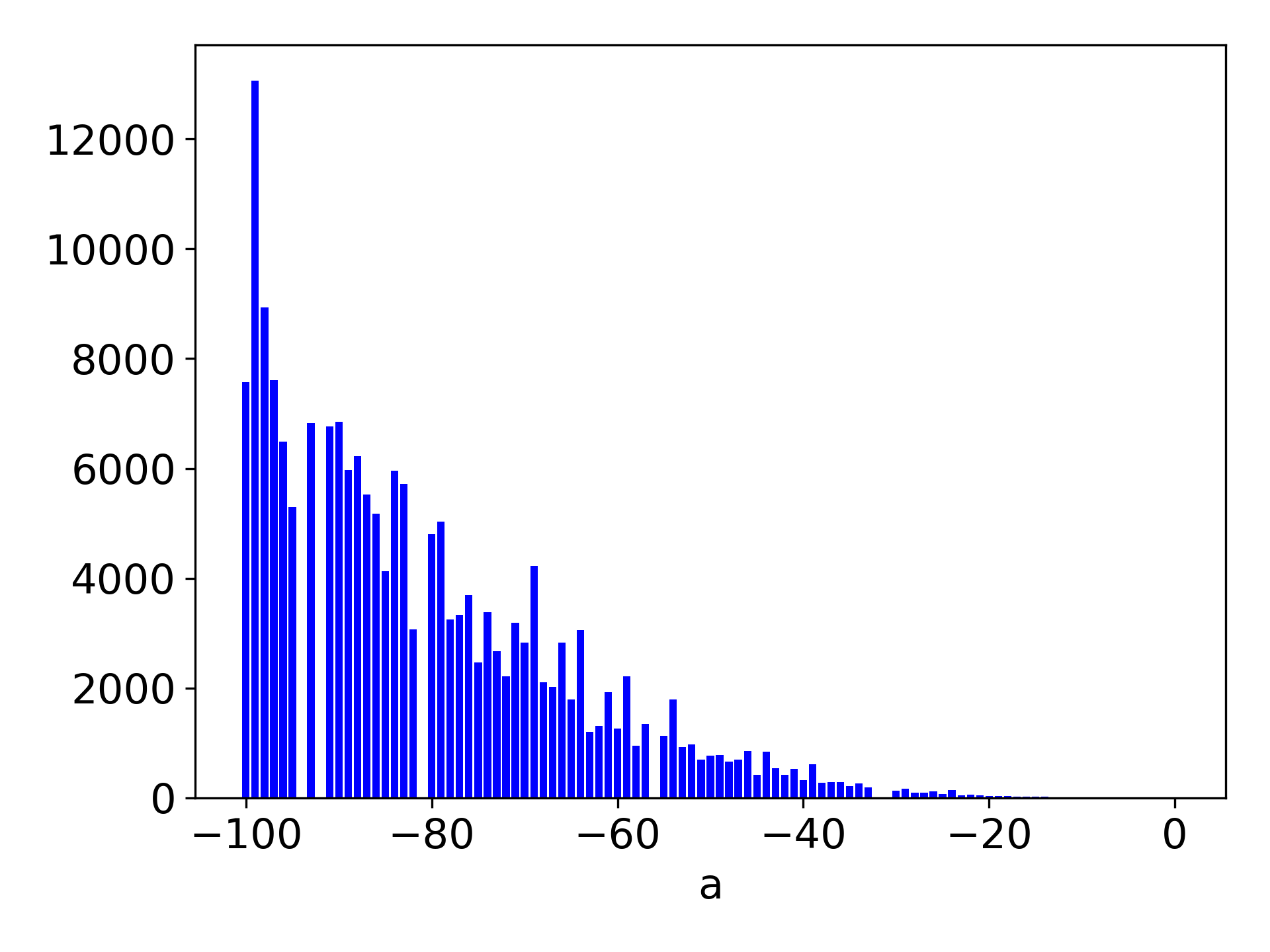}
    \caption{Orders of ideal class groups corresponding to $x^5+ax^4+(-2a-1)x^3+(2a+1)x^2+(-a+1)x-1$.}
    \label{figure:cs5}
    \end{figure}

    We give infinitely many examples of Cappell-Shaneson knot pairs that are not distinguishable by the first Alexander polynomial.

\begin{theorem}
    The following two matrices are positive Cappell-Shaneson matrices for any $l \le 0$. These matrices are not $*$-equivalent and give us inequivalent Cappell-Shaneson knot pairs with the same Alexander polynomial:
    \[
        \begin{bmatrix}
    		0 & 1 & 0 & 0 & 0\\
    		0 & 0 & 1 & 0 & 0\\
    		0 & 0 & 0 & 1 & 0\\
                0 & 0 & 0 & 0 & 1\\
                1 & 25l-7 & -50l+11 & 50l-11 & -25l+6\\
    		\end{bmatrix},
    \]
    \[
        \begin{bmatrix}
    	-2 & 5 & 0 & 0 & 0\\
            0 & 0 & 1 & 0 & 0\\
            0 & 0 & 0 & 1 & 0\\
            0 & 0 & 0 & 0 & 1\\
            -210l+55 & 525l-137 & -250l+65 & 100l-27 & -25l+8\\
    		\end{bmatrix}.
    \]

\end{theorem}
\begin{proof}
    Define $f_a(x)$ as $x^5+ax^4+(-2a-1)x^3+(2a+1)x^2+(-a+1)x-1$, which is a positive Cappell-Shaneson polynomial when $a \le 0$. Let $\theta_a$ be a root of $f_a(x)$. Then, $(5, \theta_{5^2l-6} + 2)$ is not invertible and 
    \[
    \begin{aligned}
    \{\, 5, \ \theta_{5^2l-6} + 2, \ \theta_{5^2l-6} ( \theta_{5^2l-6} + 2), \ \theta_{5^2l-6}^2 (\theta_{5^2l-6} + 2), \ \theta_{5^2l-6}^3 (\theta_{5^2l-6} + 2)\, \}
    \end{aligned}
    \]
    is a $\mathbb{Z}$-basis of $(5, \theta_{5^2l-6} + 2)$ for any $l \le 0$ since
    \begin{align*}
        f_{5^2l-6}(x)\equiv&(x+2)^2(x^3+2x+1) \pmod {5} ,\\
        f_{5^2l-6}(x) =&(x+2)(x^4+(25l-8)x^3+(-100l+27)x^2\\
                    &+(250l-65)x-525l+137)+1050l-275
    \end{align*}
    and Corollary~\ref{corollary:sourceofinffamily} can be used. The following matrix corresponds to the ideal class of $(5, \theta_{5^2l-6} + 2)$,
    \[
        \begin{bmatrix}
    	-2 & 5 & 0 & 0 & 0\\
            0 & 0 & 1 & 0 & 0\\
            0 & 0 & 0 & 1 & 0\\
            0 & 0 & 0 & 0 & 1\\
            -210l+55 & 525l-137 & -250l+65 & 100l-27 & -25l+8\\
    		\end{bmatrix},
    \]
    This matrix is not $*$-equivalent to the companion matrix corresponding to $f_{5^2l-6}(x)$ since the ideal class corresponding to the companion matrix is invertible. Therefore, the above matrix and the companion matrix give us inequivalent knot pairs that cannot be distinguished by the Alexander polynomial.
\end{proof}

\subsection{Cappell-Shaneson knot pairs for \texorpdfstring{$n=6$}{n=6}}
    We give new Cappell-Shaneson knot pairs for $n=6$. We use a positive Cappell-Shaneson polynomial $x^6+(2a+1)x^5+(a^2-a-2)x^4+(-2a^2-2a-2)x^3+(a^2-a-1)x^2+(2a+1)x+1$ with $a \le -1$, which is newly found in our paper \cite{Endo-Iwaki-Pajitnov:2025}. We can compute the orders of ideal class groups corresponding to this polynomial in the same way as for $n=4$. The order is the number of Cappell-Shaneson knot pairs with $x^6+(2a+1)x^5+(a^2-a-2)x^4+(-2a^2-2a-2)x^3+(a^2-a-1)x^2+(2a+1)x+1$ as the first Alexander polynomial.

    \begin{figure}[htbp]
    \centering
    \includegraphics[scale=0.6]{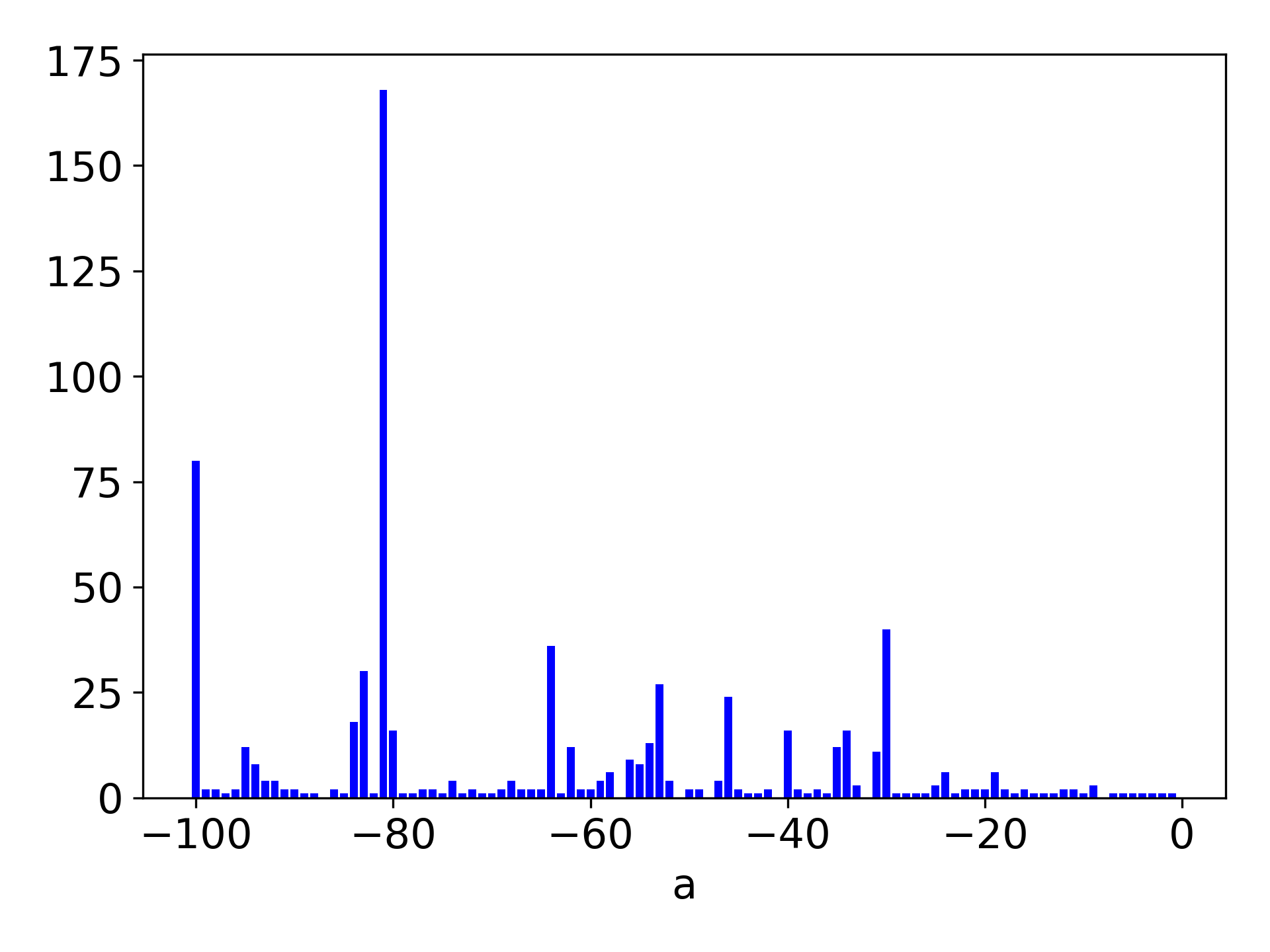}
    \caption{Orders of ideal class groups corresponding to $x^6+(2a+1)x^5+(a^2-a-2)x^4+(-2a^2-2a-2)x^3+(a^2-a-1)x^2+(2a+1)x+1$.}
    \label{figure:cs6}
    \end{figure}

\begin{theorem}
    The following two matrices are positive Cappell-Shaneson matrices for any $l \le 0$. These matrices are not $*$-equivalent and give us inequivalent Cappell-Shaneson knot pairs with the same Alexander polynomial:
    \[
        \begin{bmatrix}
    		0 & 1 & 0 & 0 & 0 & 0\\
    		0 & 0 & 1 & 0 & 0 & 0\\
    		0 & 0 & 0 & 1 & 0 & 0\\
            0 & 0 & 0 & 0 & 1 & 0\\
            0 & 0 & 0 & 0 & 0 & 1\\
            -1 & -2A_l-1 & -A_l^2+A_l+1 & 2A_l^2+2A_l+2 & -A_l^2+A_l+2 & -2A_l-1\\
    		\end{bmatrix},
    \]
    \[
        \begin{bmatrix}
    	-2 & 7 & 0 & 0 & 0 & 0\\
            0 & 0 & 1 & 0 & 0 & 0\\
            0 & 0 & 0 & 1 & 0 & 0\\
            0 & 0 & 0 & 0 & 1 & 0\\
            0 & 0 & 0 & 0 & 0 & 1\\
            B_l & C_l & D_l & E_l & F_l & -2A_l+1\\
    		\end{bmatrix},
    \] with
    \[
    A_l = 49l-8,
    \]
    \[
    B_l = -12348l^2+336l-413,
    \]
    \[
    C_l = 43218l^2-1176l+1445,
    \]
    \[
    D_l = -21609l^2+539l-715,
    \]
    \[
    E_l = 9604l^2-228l+322,
    \]
    \[
    F_l = -2401l^2+279l-104.
    \]
\end{theorem}
\begin{proof}
    Define $f_a(x)$ as $x^6+(2a+1)x^5+(a^2-a-2)x^4+(-2a^2-2a-2)x^3+(a^2-a-1)x^2+(2a+1)x+1$, which is a positive Cappell-Shaneson polynomial when $a \le -1$. Let $\theta_a$ be a root of $f_a(x)$. Then, $(7, \theta_{7^2l-8} + 2)$ is not invertible and 
    \[
    \begin{aligned}
    \{\, 7, \ &\theta_{7^2l-8} + 2, \ \theta_{7^2l-8} ( \theta_{7^2l-8} + 2), \\ &\theta_{7^2l-8}^2 (\theta_{7^2l-8} + 2), \ \theta_{7^2l-8}^3 (\theta_{7^2l-8} + 2),\ \theta_{7^2l-8}^4 (\theta_{7^2l-8} + 2)\, \}
    \end{aligned}
    \]
    is a $\mathbb{Z}$-basis of $(7, \theta_{7^2l-8} + 2)$ for any $l \le 0$ since
    \begin{align*}
        f_{7^2l-8}(x)\equiv&(x+1)(x+2)^2(x^3+x^2+x+2) \pmod {7} ,\\
        f_{7^2l-8}(x) =&(x+2)(x^5+(98l-17)x^4+(2401l^2-279l+104)x^3\\
                        &+(-9604l^2+228l-322)x^2+(21609l^2-539l+715)x\\
                        &-43218l^2+1176l-1445)+86436l^2-2352l+2891
    \end{align*}
    and Corollary~\ref{corollary:sourceofinffamily} can be used. The following matrix corresponds to the ideal class of $(7, \theta_{7^2l-8} + 2)$,
    \[
        \begin{bmatrix}
    	    -2 & 7 & 0 & 0 & 0 & 0\\
            0 & 0 & 1 & 0 & 0 & 0\\
            0 & 0 & 0 & 1 & 0 & 0\\
            0 & 0 & 0 & 0 & 1 & 0\\
            0 & 0 & 0 & 0 & 0 & 1\\
            B_l & C_l & D_l & E_l & F_l & -2A_l+1\\
    		\end{bmatrix}.
    \]
    This matrix is not $*$-equivalent to the companion matrix corresponding to $f_{7^2l-8}(x)$ since the ideal class corresponding to the companion matrix is invertible. Therefore, the above matrix and the companion matrix give us inequivalent knot pairs that cannot be distinguished by the Alexander polynomial.
\end{proof}

\subsection{Cappell-Shaneson knot pairs for \texorpdfstring{$n=7$}{n=7}}
    We give new Cappell-Shaneson knot pairs for $n=7$. We use a positive Cappell-Shaneson polynomial $x^7-ax^6+(a-2)x^5+(-a^2-1)x^4+(a^2+2)x^3+(-a+2)x^2+ax-1$ with $a \ge 0$, which is newly found in our paper \cite{Endo-Iwaki-Pajitnov:2025}. We can compute the orders of ideal class groups corresponding to this polynomial in the same way as for $n=4$. The order is the number of Cappell-Shaneson knot pairs with $x^7-ax^6+(a-2)x^5+(-a^2-1)x^4+(a^2+2)x^3+(-a+2)x^2+ax-1$ as the first Alexander polynomial.

\begin{theorem}
    The following two matrices are positive Cappell-Shaneson matrices for any $l \ge 0$. These matrices are not $*$-equivalent and give us inequivalent Cappell-Shaneson knot pairs with the same Alexander polynomial:
    \[
        \begin{bmatrix}
    		0 & 1 & 0 & 0 & 0 & 0 & 0\\
    		0 & 0 & 1 & 0 & 0 & 0 & 0\\
    		0 & 0 & 0 & 1 & 0 & 0 & 0\\
                0 & 0 & 0 & 0 & 1 & 0 & 0\\
                0 & 0 & 0 & 0 & 0 & 1 & 0\\
                0 & 0 & 0 & 0 & 0 & 0 & 1\\
                1 & -G_l & G_l-2 & -G_l^2-2 & G_l^2+1 & -G_l+2 & G_l\\
    		\end{bmatrix},
    \]
    \[
        \begin{bmatrix}
    	-6 & 11 & 0 & 0 & 0 & 0 & 0\\
            0 & 0 & 1 & 0 & 0 & 0 & 0\\
            0 & 0 & 0 & 1 & 0 & 0 & 0\\
            0 & 0 & 0 & 0 & 1 & 0 & 0\\
            0 & 0 & 0 & 0 & 0 & 1 & 0\\
            0 & 0 & 0 & 0 & 0 & 0 & 1\\
            H_l & I_l & J_l & K_l & L_l & M_l & G_l+6\\
    		\end{bmatrix},
    \] with
    \[
    G_l = 121l+20,
    \]
    \[
    H_l = 2012472l^2+1264494l+178211,
    \]
    \[
    I_l = -3689532l^2-2318239l-326720,
    \]
    \[
    J_l = 614922l^2+386353l+54450,
    \]
    \[
    K_l = -102487l^2-64372l-9072,
    \]
    \[
    L_l = 14641l^2+9922l+1445,
    \]
    \[
    M_l = -847l-174.
    \]
\end{theorem}
\begin{proof}
    Define $f_a(x)$ as $x^7-ax^6+(a-2)x^5+(-a^2-1)x^4+(a^2+2)x^3+(-a+2)x^2+ax-1$, which is a positive Cappell-Shaneson polynomial when $a \ge 0$. Let $\theta_a$ be a root of $f_a(x)$. Then, $(11, \theta_{11^2l+20} + 6)$ is not invertible and 
    \[
    \begin{aligned}
    \{\,11,\ &\theta_{11^2l+20} + 6, \ \theta_{11^2l+20} ( \theta_{11^2l+20} + 6), \ \theta_{11^2l+20}^2 (\theta_{11^2l+20} + 6), \\ &\theta_{11^2l+20}^3 (\theta_{11^2l+20} + 6), \ \theta_{11^2l+20}^4 (\theta_{11^2l+20} + 6), \ \theta_{11^2l+20}^5 (\theta_{11^2l+20} + 6)\, \}
    \end{aligned}
    \]
    is a $\mathbb{Z}$-basis of $(11, \theta_{11^2l+20} + 6)$ for any $l \ge 0$ since
    \begin{align*}
        f_{11^2l+20}(x)\equiv &(x+6)^2(x^5+x^4+3x^3+8x+7) \pmod {11} ,\\
        f_{11^2l+20}(x) = &(x+6)(x^6+(-121l-26)x^5+(847l+174)x^4\\
                        &+(-14641l^2-9922l-1445)x^3+(102487l^2+64372l+9072)x^2\\
                        &+(-614922l^2-386353l-54450)x\\
                        &+3689532l^2+2318239l+326720)\\
                        &-22137192l^2-13909434l-1960321
    \end{align*}
    and Corollary~\ref{corollary:sourceofinffamily} can be used. The following matrix corresponds to the ideal class of $(11, \theta_{11^2l+20} + 6)$,
    \[
        \begin{bmatrix}
    	-6 & 11 & 0 & 0 & 0 & 0 & 0\\
            0 & 0 & 1 & 0 & 0 & 0 & 0\\
            0 & 0 & 0 & 1 & 0 & 0 & 0\\
            0 & 0 & 0 & 0 & 1 & 0 & 0\\
            0 & 0 & 0 & 0 & 0 & 1 & 0\\
            0 & 0 & 0 & 0 & 0 & 0 & 1\\
            H_l & I_l & J_l & K_l & L_l & M_l & G_l+6\\
    		\end{bmatrix}.
    \]
    This matrix is not $*$-equivalent to the companion matrix of $f_{11^2l+20}(x)$ since the ideal class corresponding to the companion matrix is invertible. Therefore, the above matrix and the companion matrix give us inequivalent knot pairs that cannot be distinguished by the Alexander polynomial.
\end{proof}

\bibliographystyle{plain}
\bibliography{main}
\end{document}